\documentclass[ reqno]{amsart}
\usepackage{graphicx} 

\title[Weakly Compact Operators Whose Adjoints Preserve $C^*$-Structure]{Weakly Compact Operators Whose Adjoints Preserve $C^*$-Structure}
 \author[N. Hotwani]{Neha Hotwani${}^1$}
\address{Department of Mathematics\\
Shiv Nadar Institution of Eminence. Gautam Buddha
Nagar-201314, India}
\email{neha.hotwani@snu.edu.in, nehahotwani19@gmail.com}

\author[T. S. S. R. K. Rao]{T. S. S. R. K. Rao${}^2$}
\address{Department of Mathematics\\
Shiv Nadar Institution of Eminence. Gautam Buddha
Nagar-201314, India}
\email{srin@fulbrightmail.org}

\subjclass[2020]{47L07, 46B20, 46L10 (Primary);  47L05 (Secondary)}

\keywords{$C^*$-algebras, $C^*$-extreme points, Linear extreme points,  Strongly extreme points, Uniformly strongly extreme,  Spaces of operators, Geometry of Banach spaces}
\usepackage{amsmath,amsthm, amsfonts, amssymb, setspace, color, enumerate, bbold}
      \newtheorem{theorem}{Theorem}[section]

      \newtheorem{remark}[theorem]{Remark}
       \newtheorem{corollary}{Corollary}
      \newtheorem{lemma}[theorem]{Lemma}
      
      \newtheorem{proposition}[theorem]{Proposition}

      \def\N{{\mathbb N}}

      \def\cK{\mathcal K}
      
      \def\cM{\mathcal M}

      \def\cW{\mathcal W}
      
      \def\cA{\mathcal A}
      \def\cB{\mathcal B}

      \def\cK{\mathcal K}

      \def\bb1{\mathbb 1}

\usepackage{ulem}
\usepackage{titlesec}
\usepackage{xstring}

\titleformat{\section}
  {\normalfont\Large\bfseries\centering} 
  {\thesection}{1em}{}                   

\renewcommand{\thesection}{\arabic{section}}

\newcommand{\conv}{\operatorname{co}}

\usepackage{ulem}
\date{}

\usepackage[dvipsnames]{xcolor}

\usepackage[pagebackref,colorlinks,linkcolor=Maroon,citecolor=MidnightBlue,urlcolor=NavyBlue,hypertexnames=true]{hyperref}

\newtheorem{letterthm}{Theorem}

\begin{document}
\begin{abstract}
Let $\Omega$ be a compact Hausdorff space and $X$ be a complex Banach space such that $X^{**}$ is isometric to a $C^*$-algebra. Let $\cW(X^*, C(\Omega))$ denote the space of weakly compact operators. In this article, we show that the collection of operators $T$ in $\cW(X^*, C(\Omega))$ such that $T^*$ maps linear extreme points of the unit ball of $C(\Omega)^*$ to $C^*$-extreme points of $X^{**}$ is a uniformly strongly extreme set. We also investigate this phenomenon when $X^{**}$ is isometric to a space of all operators on a Banach space.
\end{abstract}
\maketitle

\section{Introduction}
    \label{sec:Intro}

    The purpose of this article is to understand the geometry of infinite-dimensional complex Banach spaces $X$ whose bidual $X^{**}$ is isometric to a $C^*$-algebra, with particular emphasis on the interplay between Banach space geometry and the underlying $C^*$-algebraic structure.

In \cite{LM}, Labuschagne et al. characterized linear maps between $C^*$-algebras whose adjoints preserve extreme structures of the dual unit balls. Inspired by this line of investigation, we study weakly compact operators between a predual of a $C^*$-algebra and a commutative unital $C^*$-algebra, whose adjoints map linear extreme points to $C^*$-extreme points of the corresponding closed unit balls.
For a given infinite compact Hausdorff space $\Omega$, we denote by $C(\Omega)$ the space of all scalar-valued continuous functions on $\Omega$, equipped with the supremum norm.
We recall that the space of weakly compact operators $\cW(X, C(\Omega))$ can be isometrically identified with the space of 
continuous functions from $\Omega$ to the dual $X^*$ , when $X^*$ is equipped with the weak topology, denoted by $WC(\Omega, X^*)$, and $WC(\Omega, X^*)$ has the supremum norm. 
This identification is given by 
$T\mapsto T^*\big|_{\delta(\Omega)}$, 
where $\delta: \Omega \to C(\Omega)^*$ denotes the Dirac map, 
that is,  $\delta(\omega)$ is the evaluation functional at $\omega$.
A similar result identifies the space of compact operators $\cK(X, C(\Omega))$ with the space of norm-continuous functions $C(\Omega, X^*)$, endowed with the supremum norm. 
Now if $X$ is such that $X^{**}$ is a $C^*$-algebra, 
we consider $\cW(X^*, C(\Omega))$ 
which now corresponds to 
$WC(\Omega, X^{**})$. 
In general, it is not known if
$WC(\Omega, X^{**})$ can be a $C^*$-algebra 
unlike the situation $C(\Omega, X^{**})$, 
which is always a $C^*$-algebra. 
Thus, the main purpose of this article is to understand the geometry of $WC(\Omega, X^{**})$ or, equivalently, the space $\cW(X^*, C(\Omega))$. 
See \cite{R} where it was shown for an infinite compact Hausdorff space $\Omega$, 
if every $X$-valued weakly continuous function is also continuous in the norm topology then $X$ has the Schur property, that is, weak and norm sequential convergence coincide in $X$.
In \cite[Lemma~3.8]{K}, it was shown that infinite-dimensional $C^*$-algebras fail the Schur property.
Thus, in our case, since $X^{**}$ fails the Schur property, $WC(\Omega, X^{**})$ does not coincide with $C(\Omega, X^{**})$. We do not know whether there exists an infinite-dimensional Banach space $ X$ such that $X^{\ast\ast}$ is a $C^\ast$-algebra and $X$ fails the Schur property.

Throughout this article, we regard $X$ as a subspace of $X^{**}$ via the canonical embedding. We use the notation $X_1$ to denote the closed unit ball of $X$ and $\partial_e X_1$ to denote the set of all linear extreme points of $X_1$ (whenever it is nonempty). A linear extreme point of $X_1$ that remains a linear extreme point of $X^{**}_1$ is called a \textbf{weak$^*$-extreme point}.
A stronger and well-known notion of extremity is that of a strongly extreme point. Recall that an element $x\in X_1$ is called a \textbf{strongly extreme point} of $X_1$ if, for every $\varepsilon>0$, there exists $\delta>0$ such that whenever $z\in X$ satisfies
\[
\|x\pm z\|\leq 1+\delta,
\]
then $\|z\|\leq \varepsilon$. For further details, we refer to \cite{DHS}.

In \cite{DHS}, Dowling et al. proved that a function $f\in C(\Omega,X)$ is a strongly extreme point of  $C(\Omega,X)_1$ if and only if $f(\omega)$ is a strongly extreme point of $X_1$ for every $\omega\in\Omega$. Later, in a seminal work, on the geometry of $WC(\Omega, X)_1$, Hu and Smith \cite{HS}
showed that extreme points need not take extreme values.
To address this phenomenon, Hu and Smith introduced the notion of a uniformly strongly extreme subset of $X_1$. A subset $D$ of $X_1$ is said to be \textbf{uniformly strongly extreme} if, for every $\varepsilon>0$, there exists $\delta>0$ such that, for all $x\in D$, if 
\[
\|x\pm z\|\leq 1+\delta
\]
for $z\in X$, then 
 $\|z\|\leq \varepsilon$.  Clearly, any subset of $D$ is a uniformly strongly extreme set. In particular, all elements of $D$ are strongly extreme points of $X_1$. They proved that $f\in WC(\Omega,X)_1$ is a strongly extreme point of $WC(\Omega,X)_1$ if and only if there exists a dense $G_\delta$ subset $\Gamma$ of $\Omega$ such that $f(\Gamma)$ is a uniformly strongly extreme subset of $X_1$ \cite[Theorem~5]{HS}. Since 
 we are interested in the operators whose adjoints map linear extreme points to $C^*$-extreme points, we need to understand the interplay between linear extremal structure and $C^*$- extremal structure. We add the adjective linear extreme while dealing with $C^*$-algebras. Recall   that (\cite{LP}) for a unital $C^*$-algebra $\cA$ with identity $\mathbf{1}$,
 an element $x \in \cA_1$ is said to be a \textbf{$C^*$-convex combination} of $k$ elements $x_1,\dots, x_k\in \cA_1$, if there exist $t_1, \dots, t_k\in \cA$ such that $\sum_{i=1}^kt_i^*t_i=\bold{1}$ and $x=\sum_{i=1}^kt_i^*x_it_i$. The $t_i$'s are known as the \textbf{coefficients} of this $C^*$- convex combination. If the coefficients, i.e., the $t_i$'s are invertible, then this $C^*$-convex combination is called a \textbf{proper $C^*$- convex combination}. 
 $x\in \cA_1$ is said to be a \textbf{$C^*$-extreme point} of $\cA_1$ if, whenever $x$ can be written as a proper $C^*$-convex combination of $x_1, \dots, x_k \in \cA_1$, that is,
\[
x= \sum_{i=1}^kt_i^*x_it_i,
\]
where $t_1,\dots, t_k\in \cA$ are invertible with $\sum_{i=1}^kt_i^*t_i=\bold{1}$, then each $x_i$ is unitarily equivalent to $x$, i.e., there exist unitaries $u_1, \dots, u_k\in \cA$ such that $x_i=u_i^*xu_i$ for $i=1, \dots, k$. With this notation, the main result of this article is the following theorem. 

\begin{letterthm}
     \label{t:A}
    Let $X$ be a Banach space such that $X^{**}$ is isometric to a $C^*$-algebra. Let $T\in \cW(X^\ast, C(\Omega))_1$ be such that $T^*$ maps linear extreme points of $C(\Omega)^*_1$ to the $C^*$-extreme points of $X^{**}_1$. Then $T$ is a strongly extreme point of $\cW(X^\ast, C(\Omega))_1$.  Moreover, the set of all such operators forms a uniformly strongly extreme subset of $\cW(X^\ast, C(\Omega))_1$.
\end{letterthm}

A proof of Theorem \ref{t:A} is given in Section \ref{sec:main-result}.
\vspace{.3cm}

 Section \ref{sec:general space} deals with extremal behavior of general spaces of operators.
 If $X$ is uniformly convex and uniformly smooth, then in Theorem~\ref{t:unif-convex}, we show that the collection of isometries and coisometries in $B(X)_1$ forms a uniformly strongly extreme subset of $B(X)_1$. One of our main results of this section is the following theorem.
\begin{letterthm}
    \label{t:B}
    Let $X$ be a uniformly convex and uniformly smooth Banach space. Let $f\in WC(\Omega, B(X))_1$ be such that $f(\omega)$ is an isometry or a coisometry in $B(X)_1$ for every $\omega\in \Omega$. Then $f$ is a strongly extreme point of $WC(\Omega, B(X))_1$. In particular, the collection of all such functions forms a uniformly strongly extreme subset of $WC(\Omega, B(X))_1$.
\end{letterthm}

A proof of Theorem \ref{t:B} is given in Section \ref{sec:general space}.

\section{Unital $C^*$-Algebras}
\label{sec:main-result}

We begin this investigation by showing that the set of all linear extreme points of the closed unit ball of a $C^*$-algebra $\cA$ is uniformly strongly extreme. For a Hilbert space $H$, let $B(H)$ denote the $C^*$-algebra of all bounded linear operators on $H$.
\begin{theorem}
    \label{t:unif-st-ext-of-A_1}
    Let  $\cA_1$ denote the closed unit ball of $\cA$. Then $\partial_e \cA_1$ is a uniformly strongly extreme subset of $\cA_1$.
\end{theorem}

\begin{proof}
 By the GNS construction, without loss of generality we can assume that $\cA$ is a $C^*$-subalgebra of $B(H)$ for some Hilbert space $H$. Let $\varepsilon>0$ be given. Fix an arbitrary $u \in \partial_e \cA_1$. Then it follows from \cite[Theorem 10.2]{T} that $u$ is a partial isometry satisfying
\[
(\bold{1}-uu^*)\cA (\bold{1}-u^*u)=\{0\}.
\]
Let $p=u^*u$ and $q=uu^*$. Then $H$ admits the decompositions
$H= pH \oplus (\bold{1}-p)H = qH \oplus (\bold{1}-q)H$.
Observe that
for any $z\in \cA$, the matrix decomposition of the operator
$z: pH \oplus (\bold{1}-p)H \to qH \oplus (\bold{1}-q)H$
is
$
\begin{bmatrix}
a & b\\
c & 0
\end{bmatrix}$,
and the matrix of $u$ is
$\begin{bmatrix}
v & 0\\
0 & 0
\end{bmatrix}$,
where $v: pH \to qH$ is unitary. We must find $\delta>0$ such that, the condition $\|u\pm z\|\leq 1+\delta$ for $z\in \cA$ implies that $\|z\|\leq \varepsilon$. Choose
\[
\delta=-1+ \sqrt{1+\frac{\varepsilon^2}{3}}.
\]
Note that $\delta>0$ and is independent of the choice of $u\in \partial_e \cA_1$. Now,
\[
u\pm z= \begin{bmatrix}
v\pm a & \pm b\\
\pm c & 0
\end{bmatrix}.
\]
Take $\xi \in pH$ with $\|\xi\|=1$. Then
\[
(u\pm z)\xi = \begin{bmatrix}
v\pm a & \pm b\\
\pm c & 0
\end{bmatrix} \begin{pmatrix}
\xi \\
0
\end{pmatrix}= \begin{pmatrix}
(v\pm a) \xi\\
\pm c \xi
\end{pmatrix}.
\]
Hence, $\|(u\pm z)\xi\|^2\leq (1+\delta)^2$ implies that
\[
\|(v\pm a)\xi\|^2+ \|c\xi\|^2\leq (1+\delta)^2.
\]
Therefore,
\begin{align}
   & \|(v+ a)\xi\|^2+ \|c\xi\|^2\leq (1+\delta)^2,
   \label{eq:v+}\\
    \text{and}~~   & \|(v - a)\xi\|^2+ \|c\xi\|^2\leq (1+\delta)^2
    \label{eq:v-}.
 \end{align}

Expanding and adding the above two inequalities \eqref{eq:v+} and \eqref{eq:v-}, one obtains
\begin{equation}
\label{eq:unif ineq}
\|v \xi\|^2+\|a \xi\|^2+ \|c \xi\|^2\leq (1+\delta)^2.
\end{equation}
Since $v: pH \to qH$ is unitary, one has $\|v\xi\|=\|\xi\|=1$. Substituting this into \eqref{eq:unif ineq}, we obtains
\[
\|a \xi\|^2 + \|c \xi\|^2\leq 2\delta+ \delta^2.
\]
Thus $\|a\|\leq \gamma$ and $\|c\| \leq \gamma$, where $\gamma= \sqrt{2\delta+ \delta^2}$. Applying the same argument to $u^*$ and $z^*$, we get $\|b\|\leq \gamma$. Therefore, if $\zeta\in H$ satisfies $\|\zeta\|=1$, then writing $\zeta=\xi+\eta$ with $\xi\in pH$ and $\eta \in (\bold{1}-p)H$, we have
\[
z\zeta= \begin{bmatrix}
a & b\\
c & 0
\end{bmatrix} \begin{pmatrix}
\xi\\
\eta
\end{pmatrix}= \begin{pmatrix}
a \xi+ b\eta\\
c \xi
\end{pmatrix}.
\]
Consequently,
\begin{align*}
\|z\zeta\|^2&=\|a \xi+b\eta\|^2+\|c\xi\|^2\\
&\leq (\gamma \|\xi\|+ \gamma \|\eta\|)^2+\gamma^2\|\xi\|^2\\
&\leq 2 \gamma^2(\|\xi\|^2+ \|\eta\|^2)+\gamma^2\|\xi\|^2\\
& \leq 2 \gamma^2(\|\xi\|^2+ \|\eta\|^2)+\gamma^2\|\xi\|^2+ \gamma^2 \|\eta\|^2\\
&= 3 \gamma^2 (\|\xi\|^2+ \|\eta\|^2)\\
&= 3 \gamma^2 \|\zeta\|^2\\
&= 3 (2 \delta+\delta^2)\\
&\leq \varepsilon^2.
\end{align*}
Thus $\|z\|\leq \varepsilon$.
This completes the proof.
\end{proof}

It is well-known that the linear extreme points of the closed unit ball of \(C(\Omega)^*\) are exactly the functionals \(\lambda \delta(\omega\)), where $\lambda$ is a scalar with \(|\lambda|=1\) and \(\delta(\omega\)) is the evaluation functional at \(\omega\in\Omega\). In the next theorem, we use this characterization to obtain some geometric analog of the corresponding results in \cite{LM}. 

\vspace{.3cm}
\begin{proof} [Proof of Theorem~\ref{t:A}]
   Since $X^{**}$ is a $C^*$-algebra (in fact, a von Neumann algebra), it follows from \cite[Theorem~B]{HR} that for every $\omega \in \Omega$ and $|\lambda|=1$, each $T^*(\lambda \delta(\omega))$ is indeed a linear extreme point of $X_1^{**}$, where $\lambda \delta(\omega)$ is defined as above.
    Let
    \[
    D= \{T^*(\lambda \delta(\omega))\in X_1^{**}: \omega \in \Omega ~ \text{and}~ |\lambda|=1\}.
    \]
Thus, $D \subseteq \partial_e X_1^{**}$. Hence, by Theorem~\ref{t:unif-st-ext-of-A_1}, the set \(D\) is uniformly strongly extreme in \(X_1^{**}\).

We now show that \(T\) is strongly extreme. Let \(\varepsilon>0\) be given. Since \(D\) is uniformly strongly extreme, there exists \(\delta>0\) such that whenever \(x\in D\) and \(y\in X^{**}\) satisfy
$\|x\pm y\|\leq 1+\delta$,
then \(\|y\|\leq \varepsilon\).
Now suppose that \(S\in \mathcal{W}(X^*,C(\Omega))\) satisfies
\[
\|T\pm S\|\leq 1+\delta.
\]
Passing to adjoints, we obtain
$\|T^*\pm S^*\|\leq 1+\delta$.
Therefore, for every \(\omega\in\Omega\) and $|\lambda|=1$,
$\|T^*(\lambda \delta(\omega))\pm S^*(\lambda \delta(\omega))\|\leq 1+\delta$.
Since \(T^*(\lambda \delta(\omega))\in D\) and \(D\) is uniformly strongly extreme, it follows that for all $\omega \in \Omega$ and 
$|\lambda|=1$,
\begin{align}
\label{eq:lin-ext}
    \|S^*(\lambda \delta(\omega))\|\leq \varepsilon.
\end{align}
Let $\conv(\partial_e C(\Omega)_1^*)$ denote the set of all convex combination of extreme points of $C(\Omega)_1^*$.
From Inequality \eqref{eq:lin-ext}, one obtains for every $\mu \in \conv(\partial_e C(\Omega)_1^*)$, one has 
\[
\|S^*(\mu)\|\leq \varepsilon.
\]
Let $\mu_0 \in C(\Omega)_1^*$ be arbitrary. Using Krein-Milman theorem, it follows that there exists a net $\mu_\alpha \subseteq \conv(\partial_e C(\Omega)_1^*)$ such that $\mu_\alpha$ converges to $\mu_0$ in weak$^*$-topology. 

Since $S$ is weakly compact, $S^*$ is weakly compact. 
In particular, $S^*$ is weak$^*$-to-weak continuous. 
Hence
$S^*(\mu_\alpha)$ converges to $S^*(\mu_0)$ weakly. 
By the weak lower semicontinuity of the norm, we obtain
\[
\|S^*(\mu_0)\| \leq \liminf_\alpha \|S^*(\mu_\alpha)\| \leq \varepsilon.
\]
As $\mu_0$ was arbitrary, 
it follows that $\|S^*\|\leq \varepsilon$. Consequently, $\|S\|\leq \varepsilon$.
This shows that \(T\) is strongly extreme.

Note that the above $\delta$ in the proof comes from Theorem \ref{t:unif-st-ext-of-A_1}, and is independent of the operator $T$. Therefore, the collection of all such operators forms a uniformly strongly extreme subset of $\mathcal{W}(X^*,C(\Omega))_1$. 
\end{proof}

\begin{corollary}
    \label{cor:unif-st-ext-wk}
    Let $X$ be a Banach space such that $X^{**}$ is isometric to  a $C^*$-algebra.
    Let $\Gamma \subseteq \Omega$ be a dense subset. Consider 
    \[
    K=\{ f\in WC(\Omega,X)_1 : f(\omega) \in \partial_eX_1~~ \text{is a weak$^*$-extreme point for all}~~ \omega \in \Gamma\}.
    \] 
    Then $ K$ is a uniformly strongly extreme subset of $WC(\Omega,X)_1$.
\end{corollary}

\begin{proof}
  For each $f\in  K$, it follows from the hypothesis that $f(\omega) \in \partial_eX_1^{**}$ 
  for every $\omega \in \Gamma$. Therefore, by Theorem \ref{t:unif-st-ext-of-A_1}, the set
\[
K_f= \{f(\omega): \omega \in \Gamma\}
\]
    is a uniformly strongly extreme subset of $X_1^{**}$. In other words, given $\varepsilon>0$ there exists $\delta>0$ such that for all $f(\omega) \in  K_f$, if $\|f(\omega) \pm x^{**}\|\leq 1+\delta$ for $x^{**} \in X^{**}$, then $\|x^{**}\|\leq \varepsilon$. 
   Observe that this $\delta$ is independent of the choice of $f\in  K$. We now show that $ K$ is uniformly strongly extreme. Fix $\varepsilon>0$, and let $\delta>0$ be the number obtained above for the set $ K_f$.
    Let $f_0\in K$ be arbitrary and suppose that $\|f_0\pm g\|\leq 1+\delta$ for some
    $g\in WC(\Omega, X)$. Then
     $\text{sup}_{\omega\in \Omega}\|f(\omega)\pm g(\omega)\|\leq 1+\delta$. In particular,
    $\text{sup}_{\omega\in \Gamma}\|f(\omega)\pm g(\omega)\|\leq 1+\delta$. 
   Since $ K_{f_0}$ is uniformly strongly extreme, it follows that $\text{sup}_{\omega\in \Gamma}\|g(\omega)\|\leq \varepsilon$.
    Now, since $g:\Omega\to X$ is weakly continuous, the map $\omega\mapsto \|g(\omega)\|$ is lower semicontinuous. Because $\Gamma$ is dense in $\Omega$, we obtain
   \[
\sup_{\omega\in\Omega}\|g(\omega)\|\le \varepsilon,
\]
that is, $\|g\|\leq \varepsilon$.
   This completes the proof.
\end{proof}

Let $\Omega'$ be a compact Hausdorff space. Consider $WC(\Omega', WC(\Omega, X))$ of weakly continuous functions from $\Omega'$ to $WC(\Omega, X)$ equipped with supremum norm. Let us note that if we consider $C(\Omega', C(\Omega, X))$, then this space can be identified with $C(\Omega \times \Omega', X)$ but similar result for $WC(\Omega', WC(\Omega, X))$ need not hold in general.  Therefore, the below Corollary \ref{c:iteration} does not follow directly from any of the previous result.
\begin{corollary}
    \label{c:iteration}
    Let $X$ be a Banach space such that $X^{**}$ is isometric to a $C^*$-algebra. Let $\Omega'$ be a compact Hausdorff space. Consider the set $K$ defined as in Corollary \ref{cor:unif-st-ext-wk}. Let $F\in WC(\Omega', WC(\Omega, X))_1$ be such that $F(\omega')\in K$ for all $\omega'\in \Omega'$. Then $F$ is a strongly extreme point of $WC(\Omega', WC(\Omega, X))_1$.
\end{corollary}

\begin{proof}
    By Corollary \ref{cor:unif-st-ext-wk}, the set $K$ is a uniformly strongly extreme subset of $WC(\Omega,X)_1$.
   Since $F(\Omega)\subseteq K$, it follows that $F(\Omega)$ is also uniformly strongly extreme. Hence, by \cite[Theorem~5]{HS}, $F$ is a strongly extreme point of $WC(\Omega', WC(\Omega, X))_1$.
\end{proof}

In \cite[Corollary~9]{HS}, Hu and Smith proved that if $f\in C(\Omega, X)_1$ is strongly extreme, then it is also strongly extreme in $WC(\Omega, X)_1$. The following corollary generalizes this result in the context of uniformly strongly extreme for the case when $X^{**}$ is a $C^*$-algebra.

\begin{corollary}
    \label{cor:unif-st-ext-cts}
    Let X be a Banach space such that $X^{**}$ is isometric to  a $C^*$-algebra. Consider
    \[
  L=  \{f\in C(\Omega, X)_1 : f~~ \text{is strongly extreme}\}.
    \]
    Then $L$ is a uniformly strongly extreme subset of $WC(\Omega, X)_1$.
\end{corollary}

\begin{proof}
    For each $f\in L$, $f$ is a strongly extreme point of $C(\Omega, X)_1$. Hence, by \cite{DHS}, $f(\omega)$ is strongly extreme for every $\omega \in \Omega$. Moreover, using the fact that every strongly extreme point of $X_1$ remains strongly extreme in $X_1^{**}$, and that $X^{**}$ is a $C^*$-algebra, it follows from \cite[Theorem~B]{HR} and Theorem \ref{t:unif-st-ext-of-A_1} that the set
\[
\{f(\omega): \omega\in\Omega\}
\]
    is a uniformly strongly extreme subset of $X_1^{**}$.
    Applying the same argument as in Corollary \ref{cor:unif-st-ext-wk}, one concludes that the set $L$ is uniformly strongly extreme.
\end{proof}

The next proposition shows that every uniformly strongly extreme subset of $X_1$ remains uniformly strongly extreme when regarded as a subset of $X_1^{**}$.

\begin{proposition}
    \label{p:unif-remain-in-bidual}
    Let $X$ be a Banach space. Let $D\subseteq X_1$ be a uniformly strongly extreme subset. Then $D$ is also a uniformly strongly extreme subset of  $X_1^{**}$. 
\end{proposition}

\begin{proof}
Let $J: X \to X^{**}$ denote the canonical embedding.
Let \(\varepsilon>0\) be arbitrary. Choose \(0<\rho<\varepsilon\). Since \(D\) is a uniformly strongly extreme subset of \(X_1\), there exists \(\delta>0\) such that for every \(x\in D\), whenever 
$\|x\pm y\|\le 1+\delta$ for $y\in X$ implies  $\|y\|\le \rho$.
Fix \(\gamma>0\) so that
\[
(1+\gamma)\left(1+\frac{\delta}{2}\right)<1+\delta
\qquad\text{and}\qquad
(1+\gamma)\rho<\varepsilon.
\]
Choose \(\delta'=\delta/2\).
Now, let \(x_0\in D\) be arbitrary, and let \(y^{**}\in X^{**}\) be such that
\[
\|J(x_0)\pm y^{**}\|\le 1+\delta'.
\]
Consider the finite-dimensional subspace
\[
E=\operatorname{span}\{J(x_0),y^{**}\}\subset X^{**}.
\]
By the Principle of Local Reflexivity (\cite{JL}), there exists an injective linear operator \(T:E\to X\) satisfying
\[
T(J(x_0))=x_0
\qquad\text{and}\qquad
\|T\|,\ \|T^{-1}\|\le 1+\gamma.
\]
Let \(y_0=T(y^{**})\). Then
\[
\|x_0\pm y_0\|
=\|T(J(x_0)\pm y^{**})\|
\le \|T\|\,\|J(x_0)\pm y^{**}\|
\le (1+\gamma)\left(1+\delta'\right)
<1+\delta.
\]
Hence, by the choice of \(\delta\),
$\|y_0\|\le \rho$.
Consequently,
\[
\|y^{**}\|
=\|T^{-1}(y_0)\|
\le \|T^{-1}\|\,\|y_0\|
\le (1+\gamma)\rho
<\varepsilon.
\]
Thus, whenever for $x\in D$,
\(
\|J(x)\pm y^{**}\|\le 1+\delta'\) for $y^{**}\in X^{**}$, we have
$\|y^{**}\|\le \varepsilon$.
This proves that \(D\) is a uniformly strongly extreme subset of \(X^{**}_1\).
\end{proof}

\section{General Spaces of Operators}
\label{sec:general space}
In this section, we prove Theorem \ref{t:B}.
Recall that a Banach space $X$ is \textbf{uniformly convex} if, for every $\varepsilon >0$ there exists $\delta >0$ such that whenever $\|x\|,\|y\|\leq 1$ and $\|x-y\| \geq \varepsilon$, one has $\|x+y\|\leq 2(1-\delta)$. A Banach space $X$ is \textbf{uniformly smooth} if its dual space $X^*$ is uniformly convex. See \cite{D}. 
For $1<p<\infty$, $L_p(\mu)$ spaces are well-known examples of spaces that are both uniformly convex and uniformly smooth.
An operator $T \in B(X)$ is called a coisometry if $T^\ast \in B(X^\ast)$ is an isometry. 
 The following lemma
 is easy to see. For completeness, we provide a proof.
\begin{lemma}
    \label{lem:unit-sphere-unif-st-ext}
    Let $X$ be a uniformly convex Banach space. Then the set 
    \[
    S_X=\{ x\in X_1: \|x\|=1\}
    \]
    is a uniformly strongly extreme subset of $X_1$.
\end{lemma}

\begin{proof}
Assume, to the contrary, that $S_X$ is not uniformly strongly extreme. Then there exists $\varepsilon>0$ such that, for every $n\in \N$, there are $x_n\in S_X$ and $y_n \in X$ with
\[
\|x_n\pm y_n\|\leq 1+\frac{1}{n}~~ \text{and}~~ \|y_n\|>\varepsilon.
\]
 Let
\[
u_n=\frac{x_n+y_n}{1+\frac{1}{n}} ~~ \text{and}~~ v_n= \frac{x_n-y_n}{1+\frac{1}{n}}.
\]
Note that $\|u_n\|,\|v_n\|\leq 1$ and $\|u_n - v_n\| > \frac{2\varepsilon}{1+\frac{1}{n}} > \varepsilon$. Moreover,
\[
\left\|\frac{u_n+v_n}{2}\right\| = \frac{1}{1+\frac{1}{n}} \to 1.
\]
This contradicts the uniform convexity of $X$. Hence, $S_X$ is a uniformly strongly extreme subset of $X_1$.
\end{proof}


\begin{theorem}
    \label{t:unif-convex}
    Let $X$ be a uniformly convex Banach space. Then the set
    \[
    D=\{T\in B(X)_1: T ~\text{is an isometry in}~ B(X)\}
    \]
    is a uniformly strongly extreme subset of $B(X)_1$. Similarly, if $X^*$ is uniformly convex, then the set of coisometries in $B(X)$ also forms a uniformly strongly extreme subset of $B(X)_1$.
\end{theorem}

\begin{proof}
    Let $\varepsilon>0$ be given. Since $X$ is a uniformly convex Banach space, from Lemma \ref{lem:unit-sphere-unif-st-ext}, it follows that the unit sphere $S_X$ is a uniformly strongly extreme subset of $X_1$. Therefore, for the given $\varepsilon>0$, there exists $\delta>0$ such that for all $x\in S_X$, if $\|x\pm y\|\leq 1+\delta$ for $y\in X$ then 
    \begin{equation}
        \|y\|\leq \varepsilon.
    \end{equation}
 We use this same $\delta>0$ to show that $D$ is a uniformly strongly extreme subset of $B(X)_1$. To this end, let $T\in D$ be arbitrary, and suppose that $\|T\pm S\|\leq 1+\delta$ for some $S\in B(X)$. Then $\sup_{\|x\|=1}\|T(x)\pm S(x)\|\leq 1+\delta$. In particular,
for every $x\in S_X$, we have $\|T(x)\pm S(x)\|\leq 1+\delta$. Since T is an isometry, $T(x)\in S_X$. Hence, by the uniform strong extremality of $S_X$, it follows that $\|S(x)\|\leq \varepsilon$ for every $x\in S_X$. Therefore, $\|S\|\leq \varepsilon$. 
This shows that $D$ is a uniformly strongly extreme subset of $B(X)_1$. 
\end{proof}

Now we prove Theorem \ref{t:B}.
\begin{proof}
Let $K$ denote the set of all isometries and coisometries in $B(X)_1$. It follows from Theorem \ref{t:unif-convex} that $K$ is a uniformly strongly extreme subset of $B(X)_1$.
Since $f(\omega)$ is either an isometry or a coisometry in $B(X)_1$, for every $\omega \in \Omega$, it follows that $f(\Omega)\subseteq K$. Hence $f(\Omega)$ is a uniformly strongly extreme subset of $B(X)_1$. By applying \cite[Theorem~5]{HS}, one obtains that $f$ is a strongly extreme point of $WC(\Omega, B(X))_1$. 
Moreover, by the same argument as in Corollary \ref{cor:unif-st-ext-wk}, we conclude that the set of all such functions forms a uniformly strongly extreme subset of $WC(\Omega, B(X))_1$. This completes the proof.
\end{proof}

\begin{remark}
\label{rem:lp-case}
     Let $1<p<\infty$, $p \neq 2$. Hennefeld in \cite{H} exhibits compact operators that are extreme points of ${ B}(\ell^p)_1$ that clearly are neither isometries or coisometries. To see this, in the notation of \cite{H},  note that towards the end of the proof of Proposition 2.3,  to check the extreme point criterion, one does not need the operator $B \in {B}(\ell^p)_1$ to be compact. Thus, the operators $S+V+T$ are compact extreme points in ${ B}(\ell^p)_1$. As ${\mathcal K}(\ell^p)^{\ast\ast} = { B}(\ell^p)$, we get that these are weak$^\ast$-extreme points in ${\mathcal K}(\ell^p)_1$.  It can be shown that these extreme points are not strongly extreme.  An open problem in this area is to understand the geometric structure of a Banach space $X$ such that $X^{\ast\ast}$ is isometric to ${ B}(\ell^p)$.

\end{remark}

Using the fact that every Hilbert space is a uniformly convex and uniformly smooth Banach space, one has the following proposition.

\begin{proposition}
    \label{p:Hilbert space case}
     Let $\{H_\alpha\}_{\alpha \in I}$ be a family of Hilbert spaces, and let
       $X=\oplus_\infty B(H_\alpha)$ $(\ell_\infty$-sum). 
    Suppose that $f \in WC(\Omega, X)_1$ satisfies $f(\omega) \in \partial_e X_1$  for every $\omega\in \Omega$. Then $f$ is a strongly extreme point of $WC(\Omega, X)_1$. Moreover, the set
    \[
  L=  \{f\in WC(\Omega, X)_1: f(\omega) \in \partial_e X_1 ~\text{for all}~ \omega \in \Omega\}
    \]
    is a uniformly strongly extreme subset of $WC(\Omega, X)_1$.
\end{proposition}

\begin{proof}
Since for each $\omega \in \Omega$, $f(\omega) \in X_1$ is a linear extreme point, it follows that for each $\alpha \in I$, $f(\omega)(\alpha)$ is a linear extreme point of $B(H_\alpha)_1$. In particular, each $f(\omega)(\alpha)$ is either an isometry or a coisometry in $B(H_\alpha)$. 
Now using the same argument as in Theorem \ref{t:unif-convex}, we have that 
the set 
\[
\{f(\omega) : \omega \in \Omega\}
\]
is a uniformly strongly extreme subset of $X_1$. Applying \cite[Theorem~5]{HS}, we obtain that $f$ is a strongly extreme point of $WC(\Omega, X)_1$. The proof that the set $L$ is uniformly strongly extreme now follows in the same manner as in the proof of Theorem~\ref{t:A}. This completes the proof.
\end{proof}

The following corollary is a special case of Proposition \ref{p:Hilbert space case}.

\begin{corollary}
    \label{c:vN}
    Let $\cM$ be a von Neumann algebra whose predual has the Radon Nikodym property (RNP). Then the set
    \[
    \{f\in WC(\Omega, \cM)_1: f(\omega)\in \partial_e \cM_1 ~\text{for all}~\omega\in \Omega\}
    \]
    is a uniformly strongly extreme subset of $WC(\Omega, \cM)_1$.
\end{corollary}

\begin{proof}
By \cite{C}, there exists a family of Hilbert spaces
\(\{H_\alpha\}_{\alpha \in I}\) such that
$\mathcal{M}$ is $*$-isomorphioc to $\oplus_{\infty} B(H_\alpha)$.
hence, the conclusion follows from Proposition~\ref{p:Hilbert space case}.    
\end{proof}

The next result gives the commutative case of Corollary~\ref{cor:unif-st-ext-wk}. Recall that Banach spaces whose biduals are commutative von Neumann algebras are called $L^1$-predual spaces; see \cite{L} for more details. We need the following property of an $L_1$-predual space $X$ from \cite{HL}. If $x_0\in \partial_e X_1$ and $x^*\in \partial_e X_1^*$, then $|x^*(x_0)|=1$.
\begin{theorem}
\label{t:unif}
Let $X$ be a Banach space such that $X^{**}$ is isometric to $C(\Omega)$ for some compact Hausdorff space $\Omega$.
Let $\Omega'$ be another compact Hausdorff space, and let $\Gamma \subseteq \Omega'$ be dense in $\Omega'$. 
Then the set 
\[
K=\{ f\in WC(\Omega', X)_1 : f(\omega) \in \partial_e X_1 ~ \text{for all}~ \omega\in \Gamma\}
\]
is a uniformly strongly extreme subset of $WC(\Omega', X)_1$.
\end{theorem}

\begin{proof}
    Since $X^{**}$ is $C(\Omega)$ and for each $f\in K$ and $\omega \in \Gamma$,  $f(\omega)$ is an extreme point of $X_1$, it follows easily from the result quoted above that $f(\omega)$ is a strongly extreme point. Using the fact that strongly extreme points remain strongly extreme in bidual, one obtains that $f(\omega)$ is a strongly extreme point of $X_1^{**}$. Hence, Theorem~\ref{t:unif-st-ext-of-A_1} implies that $\{ f(\omega): \omega \in \Gamma\}$ is a uniformly strongly extreme subset of $X_1^{**}$. Applying the same argument as in Corollary~\ref{cor:unif-st-ext-wk}, we conclude that $K$ is a uniformly strongly extreme subset of $WC(\Omega', X)_1$. This completes the proof.
\end{proof}

As a further application to the non self-adjoint case, we point out the following proposition.
Let $\Omega$ be a compact Hausdorff space and $\cB \subset C(\Omega)$ be a closed subalgebra containing constants and separating points of $\Omega$. Without loss of generality, we may assume that $\Omega$ is the Shilov boundary of $\cB$. See \cite{ES} for further details.
In \cite{R2}, the author proved that if $f$ is a linear extreme point of $\cB_1$, which is also a weak$^{*}$-extreme point, then $f$ is in fact a strongly extreme point of $\cB_1$. Moreover, the author also proved that $|f|\equiv 1$ on $\Omega$. Hence, the following proposition is easy to see.

\begin{proposition}
    \label{p:unif-alg}
    Let $\cB$ be a uniform algebra. Then the set
    \[
    D= \{ f\in WC(\Omega, \cB)_1: f(\omega)~ \text{is a weak$^{*}$-extreme point in} ~ \cB_1 ~\text{for all}~\omega\in \Omega\}
    \]
    is a uniformly strongly extreme subset of $WC(\Omega, \cB)_1$.
\end{proposition}

The concluding result deals with the topological nature of the set of extremal points. 

\begin{corollary}
    \label{c:topo-nature}
    Let $X$ be an $L^1$-predual space. Then $\partial_e X_1$ (when non-empty) is a weakly closed set.
    In particular, the set
    \[
    L=\{ f\in WC(\Omega, X)_1: f(\omega) \in \partial_e X_1 ~\text{for all}~\omega\in \Omega\}
    \]
    is a weakly closed set.
 \end{corollary}

 \begin{proof}
     If $x\in \partial_e X_1$, then it follows from the results in \cite{HL}, for each $x^*\in \partial_e X_1^*$, $|x^*(x)|=1$. So, if $\{x_\alpha\}_{\alpha \in I} \subseteq \partial_e X_1$ is a net and $x_\alpha$ converges weakly to $x_0$ in $X_1$. Then for any  $x^*\in \partial_e X_1^*$, $|x^*(x_0)|=1$. But this implies $x_0\in \partial_e X_1$. Similarly, if $\{f_\alpha\}\subseteq L$, $f\in WC(\Omega, X)_1$ and $f_\alpha$ converges to $f$ weakly, let $x^*\in \partial_e X_1^*$, $\omega \in \Omega$,  consider the functional $(\delta(\omega) \otimes x^*)(f_\alpha)=x^*(f_\alpha(\omega))$ this will converge to $(\delta(\omega) \otimes x^*)(f)=x^*(f(\omega))$. Now, as we observed earlier since $|x^*(f_\alpha(\omega))|=1$ for all $\alpha$, we obtain $|x^*(f(\omega))|=1$ for all $x^*$. In particular, $f(\omega)\in \partial_e X_1$. Hence $f\in L$.
 \end{proof}

 \begin{remark}
     We do not know the corresponding $C^*$-algebra version of Corollary 5.
 \end{remark}

\vspace{.3cm}
\textbf{Acknowledgements:} This work is part of the project ``Classification of Banach spaces using differentiability",
funded by the Anusandhan National Research Foundation (ANRF), Core Research Grant, CRG2023-000595.
The first author is a research associate in this project. 

\vspace{.3cm}
\textbf{Declaration:}
No conflicts of interest.

\bibliographystyle{alpha}
\bibliography{references}
\end{document}